\documentclass[12pt]{article}
\usepackage{amsmath,amsfonts,amssymb,amsthm,color}
\usepackage{txfonts}%
\usepackage[dvipdfmx]{graphicx}
\usepackage{longtable}

\newtheorem{theorem}{Theorem}[section]

\newtheorem{proposition}[theorem]{Proposition}
\newtheorem{lemma}[theorem]{Lemma}

\newtheorem{corollary}[theorem]{Corollary}

\renewcommand{\indent}{\hspace*{5mm}}
\newcommand{\trns}{'}
\newcommand{\thetat}{\theta\trns}
\newcommand{\ckn}{u}

\newcommand{\cK}{{\cal K}}
\newcommand{\cP}{{\cal P}}

\newcommand{\cV}{{\cal V}}

\newcommand{\cS}{{\cal S}}
\newcommand{\cT}{{\cal T}}

\newcommand{\mle}{\hat\theta_n^*}

\newcommand{\grad}{\operatorname{grad}}

\begin{document}

\title{Bayesian logistic betting strategy against probability forecasting}

\author{
Masayuki Kumon\thanks{Japanese Association for Promoting Quality Assurance in Statistics},
Jing Li\thanks{Graduate School of Information Science and Technology,  University of Tokyo},
Akimichi Takemura\footnotemark[2] 
and Kei Takeuchi\thanks{Emeritus, Faculty of Economics, University of Tokyo}
}
\date{April 2012}
\maketitle

\begin{abstract}
We propose a betting strategy  
based on Bayesian logistic
regression modeling 
for the probability forecasting game 
in the framework of game-theoretic probability by Shafer and Vovk \cite{ShaVov01}.  
We prove some results concerning the strong law of large numbers in the probability forecasting
game with side information based on our strategy.
We also apply our strategy for assessing the quality of 
probability forecasting by the Japan Meteorological Agency.
We find that our strategy beats the agency by exploiting its tendency of avoiding clear-cut forecasts.
\end{abstract}

\noindent
{\it Keywords and phrases:} \ 
exponential family,
game-theoretic probability,
Japan Meteorological Agency,
probability of precipitation,
strong law of large numbers.

\section{Introduction}
\label{sec:intro}

In this paper we consider assessing quality of probability forecasting for binary
outcomes. A primary example of probability forecasting is the
probability of precipitation announced by weather forecasting
agencies. The binary outcomes are either ``rain'' (more precisely,  precipitation
above certain amount during a specified period at a particular location) or ``no rain''. In the
United States the National Weather Service started to announce
probability of precipitation in 1965 (cf.\ \cite{dawid-1986}), whereas
the Japan Meteorological Agency started probability forecasting in
1980 for Tokyo area and extended it to the whole Japan 
in 1986\footnote{{\tt http://www.jma.go.jp/jma/kishou/intro/gyomu/index2.html} (in Japanese)}.
How can we assess the quality of probability
forecasting?  We propose to assess probability forecasting by setting
up a hypothetical betting game against forecasting agencies in the
framework of game-theoretic probability by Shafer and Vovk
\cite{ShaVov01}.

We can regard the capital process of a betting
strategy as a test statistic of a statistical hypothesis (\cite{shafer-etal-test-martingales},
\cite{ktt-new-procedure}).
Our null hypothesis is that 
given the probability $p_n$ announced by the agency, the outcome is indistinguishable
from the Bernoulli trial with success probability $p_n$.  If this hypothesis is true, then
the capital process becomes a non-negative martingale and the capital process 
converges to a finite value %
almost surely.
However if the announced probability $p_n$ is not good, then a clever strategy
may be able to beat the 
forecasting agency in the betting game. In our game we construct a betting strategy based on
Bayesian logistic regression modeling, which is a very standard statistical model for analyzing
binary responses.  We will prove some results on the strong law of large numbers in probability forecasting game
with side information based on our betting strategy.
We also see that our strategy works well against probability of precipitation announced by 
the Japan Meteorological Agency.

Organization of this paper is as follows.
In Section \ref{sec:formulation} we formulate
the probability forecasting game with side information and derive some basic properties of
betting strategies.  It also serves as a brief introduction
to game-theoretic probability theory.
In Section 
\ref{sec:logistic-betting} we introduce our betting strategy based on
logistic regression model.
In Section \ref{sec:forcing} we prove some properties of our logistic betting
strategy in the framework of game-theoretic probability.
In Section \ref{sec:experiments} we give numerical studies of our strategy.
In particular we apply our strategy to the data on probability of precipitation
announced by the Japan Meteorological Agency. 
We end the paper with some discussions in Section \ref{sec:summary}.

\section{Formulation of the probability forecasting game and summary of preliminary results}
\label{sec:formulation}

In this section we formulate the probability forecasting game and extend it to include
side information.  We mostly follow the results in \cite{ktt-2008saa}. 

At the beginning of day $n$ (or at the end of day $n-1$) an agency (we call it ``Forecaster'')
announces a probability
$p_n$ of certain event in day $n$, such as precipitation in day $n$.
Let $x_n=0,1$ be the indicator variable for the event, i.e., $x_n=1$ if the event occurs and
$x_n=0$ otherwise.   We suppose that a player ``Reality'' decides the binary outcome $x_n$.
When Forecaster announces $p_n$, it also sells a ticket
with the price of $p_n$ per ticket.  The ticket pays one  monetary unit when the event occurs
in day $n$, i.e., the value of the ticket at the end of day $n$ is $x_n$.
A bettor or gambler, called ``Skeptic'', buys $M_n$ tickets with the price of $p_n$ per ticket.
Then the payoff to Skeptic in day $n$ is $M_n (x_n - p_n)$.
We allow $M_n$ to be negative, so that Skeptic can bet also on the non-occurrence of the
event.  If the probability announced by the agency is appropriate, it is hard for Skeptic
to make money in this game.  On the other hand, if the probability is biased in some way, Skeptic
may be able to increase his capital denoted by $\cK_n$.   Hence we can measure
the quality of probability forecasting in terms of $\cK_n$.  

We now give a protocol of the game, following the notational convention of
\cite{ShaVov01}.

\begin{quote}
{\sc Binary Probability Forecasting} (BPF)\\
\textbf{Protocol}:\\
\indent Skeptic announces his initial capital $\cK_0=1$.\\
\indent FOR $n=1,2,\ldots$:\\
\indent\indent Forecaster announces $p_n\in (0,1)$.\\
\indent\indent Skeptic announces $M_n\in\mathbb{R}$.\\
\indent\indent Reality announces $x_n\in \{0,1\}$.\\
\indent\indent $\cK_n:=\cK_{n-1}+M_n(x_n-p_n)$.\\
\textbf{Collateral Duty}:
Skeptic must keep $\cK_n$ non-negative.
\end{quote}

Forecaster is supposed to decide its forecast $p_n$ based on all relevant
side information available at the time of announcement.   
We modify the above protocol so that Forecaster also discloses the relevant side information $c_n$, 
which is a $d$-dimensional  column vector, 
together with the probability $p_n$. 
Furthermore we define auxiliary capital processes  $\cS_n$ and $\cV_n$.

\begin{quote}
{\sc Binary Probability Forecasting With Side Information} (BPFSI)\\
\textbf{Protocol}:\\
\indent $\cK_0:=1, \cS_0:=0, \cV_0:=0$.\\
\indent FOR $n=1,2,\ldots$:\\
\indent\indent Forecaster announces $p_n\in (0,1)$ and $c_n\in \mathbb{R}^d$.\\
\indent\indent Skeptic announces $M_n$.\\
\indent\indent Reality announces $x_n\in \{0,1\}$.\\
\indent\indent $\cK_n:=\cK_{n-1}+M_n(x_n-p_n)$.\\
\indent\indent $\cS_n:=\cS_{n-1}+c_n(x_n-p_n)$.\\
\indent\indent $\cV_n:=\cV_{n-1}+ c_n c_n\trns \; p_n(1-p_n)$. \\
\textbf{Collateral Duty}:
Skeptic must keep $\cK_n$ non-negative.
\end{quote}

In the protocol, $c_n\trns$ denotes the transpose of $c_n$, 
$\cK_n$ is a scalar, $\cS_n$ is a $d$-dimensional column vector and
$\cV_n$ is a $d\times d$ symmetric matrix.

If $d=1$ and $c_n\equiv 1$, then $\cS_n = \sum_{i=1}^n (x_i - p_i)$.  When we study
the usual strong law of large numbers in game-theoretic probability, we are interested in
the convergence
$\cS_n/n \rightarrow 0$ as $n\rightarrow\infty$.
Generalizing this case, in the presence of side information, we are
interested in the convergence 
$\cV_n^{-1} \cS_n %
\rightarrow 0$, although the order of $\cV_n$ may be different from $O(n)$.
We call this convergence the usual form
of the strong law of large numbers in BPFSI.  See Theorem \ref{thm:usual-SLLN} in 
Section \ref{subsec:usual-SLLN}.
However, 
as we prove in Theorem \ref{thm:main} of Section \ref{subsec:main-result},
under mild regularity conditions, we can prove a stronger result 
\[
\lim_n g(\cV_n)^{-1} \cS_n =0,
\]
where $g(\cV)$ is close to  $\cV^{1/2}$ such as
$g(\cV)=\cV^{1/2+\epsilon}$, $\epsilon>0$.  

Let 
\[
\nu_n=\frac{M_n}{\cK_{n-1}}
\]
denote the fraction of the capital Skeptic bets on day $n$. Then the capital process $\cK_n$ is written as
\begin{equation}
\label{eq:ckn}
\cK_n = \prod_{i=1}^n (1+\nu_i(x_i - p_i)).
\end{equation}

Now suppose that Skeptic himself models Reality's move
as a Bernoulli variable with  the success probability $\hat p_n \in (0,1)$.
If Skeptic totally trusts Forecaster, then he sets $\hat p_n=p_n$.
However if Skeptic does not totally trust Forecaster he may formulate
$\hat p_n$  differently from $p_n$.
Furthermore suppose that Skeptic uses the ``Kelly criterion'' (\cite{kelly-criterion-book}, \cite{kelly-1956})
to determine $\nu_n$ so as to maximize the expected value of the logarithm of the capital growth under $\hat p_n$:
\[
\nu_n \ : \ E_{\hat p_n}[\log(1+\nu(x_n-p_n))]
\ \rightarrow \ \max.
\]
Writing
\[
E_{\hat p_n}[\log(1+\nu(x_n-p_n))]
=\hat p_n \log (1+\nu(1-p_n)) + (1-\hat p_n) \log (1-\nu p_n)
\]
and differentiating this with respect to $\nu$, 
the unique maximizer $\nu_n$ is obtained as
\begin{equation}
\label{eq:nu_n}
\nu_n = \frac{\hat p_n - p_n}{p_n(1-p_n)} = \frac{\hat p_n}{p_n} - \frac{1-\hat p_n}{1-p_n}.
\end{equation}
With this choice of $\nu_n$ we have
\begin{align*}
1+\nu_n (x_n-p_n) &= \begin{cases} \hat p_n/p_n  & \text{if} \ x_n=1\\
                                (1-\hat p_n)/(1-p_n) & \text{if}\ x_n=0
                              \end{cases}\\
&=\frac{\hat p_n^{x_n} (1-\hat p_n)^{1-x_n}}{p_n^{x_n} (1-p_n)^{1-x_n}} .
\end{align*}                              
Hence \eqref{eq:ckn} is written as
\[
\cK_n = \frac{\prod_{i=1}^n \hat p_i^{x_i} (1-\hat p_i)^{1-x_i}}
        {\prod_{i=1}^n p_i^{x_i} (1-p_i)^{1-x_i}}.
\]
    
In the case that Skeptic models the joint probability $\hat p(x_1, \dots, x_n)$ of Reality's moves, 
$\hat p_n$ is given as the conditional probability
\[
\hat p_n = \frac{\hat p(x_1, \dots, x_{n-1},1)}{\hat p(x_1,\dots, x_{n-1})}.
\]
In this case
\[
\hat p_n^{x_n} (1-\hat p_n)^{1-x_n} = \frac{\hat p(x_1, \dots, x_{n-1},x_n)}{\hat p(x_1,\dots, x_{n-1})}, \qquad
x_n=0,1,
\]
and  $\cK_n$ is written as
\begin{equation}
\label{eq:joint-probability-formulation}
\cK_n = \frac{\hat p (x_1, \dots, x_n)}{\prod_{i=1}^n p_i^{x_i}(1-p_i)^{1-x_i}}.
\end{equation}

For the rest of this section we introduce some terminology of game-theoretic probability.
An infinite sequence  of Forecaster's moves and Reality's moves
\[
\xi=p_1 c_1 x_1 p_2 c_2 x_2 \dots
\]
is called a {\em path}.  The set $\Omega$ of all paths  is called the {\em sample space}.
A subset $E\subset \Omega$ is an {\em event}.  A {\em strategy} $\cP$ of Skeptic 
determines $\hat p_n$ based on a partial path $p_1 c_1 x_1\dots p_{n-1} c_{n-1} x_{n-1} p_n c_n$:
\[
\cP: p_1 c_1 x_1\dots p_{n-1} c_{n-1} x_{n-1} p_n c_n \mapsto \hat p_n, \qquad n=1,2,\dots.
\]
$\cK_n^{\cP}=\cK_n^{\cP}(\xi)$ denotes the capital process when Skeptic adopts the strategy $\cP$.
We say that Skeptic can {\em weakly force} an event $E$ by a strategy $\cP$ if 
$\cK_n^{\cP}$ is never negative and 
\[
\limsup_n \cK_n^{\cP}(\xi) = \infty   \quad \forall \xi\not\in E.
\]
For two events $E_1, E_2 \subset \Omega$, $E_1^C \cup E_2$ is denoted as $E_1 \Rightarrow E_2$, where
$E_1^C$ is the complement of $E_1$.
We say that by a strategy $\cP$ Skeptic can weakly force a conditional event $E_1\Rightarrow E_2$ if
$\cK_n^{\cP}$ is never negative and
\[
\limsup_n \cK_n^{\cP}(\xi) = \infty   \quad \forall \xi \in  E_1\cap E_2^C.
\]
$E_1$ is interpreted as a set of regularity conditions for the event $E_2$ to hold.

Let $\lambda_{\max,n}$ and $\lambda_{\min,n}$ denote the maximum and the minimum
eigenvalues of $\cV_n$. 
In this paper we consider the following regularity conditions: 
\begin{itemize}
\setlength{\itemsep}{0pt}
\item[i)] $\lim_n \lambda_{\min,n}=\infty$.
\item[ii)] $\limsup_n \lambda_{\max,n}/\lambda_{\min,n} < \infty$.
\item[iii)] $\{ c_1,c_2,\dots\}$ is a bounded set.
\end{itemize}
Namely we take $E_1$ as
\begin{equation}
\label{eq:our-regularity}
E_1 = \{\xi \mid \lim_n \lambda_{\min,n}=\infty, \limsup_n \lambda_{\max,n}/\lambda_{\min,n} < \infty \ \text{and}\ 
c_1,c_2,\dots \text{are bounded}\} .
\end{equation}
The condition i) makes the meaning of ``$\cV_n \rightarrow \infty$'' precise.
The condition ii) means that $\cV_n$ stays away from being singular.  For $d=1$ ii) is trivial
and not needed.

\section{Logistic betting strategy}
\label{sec:logistic-betting}

In this section we introduce a betting strategy based on logistic modeling of Reality's moves.

As in the previous section 
Skeptic models $x_n$ as a Bernoulli variable with the success probability $\hat p_n$. 
Furthermore we specify that Skeptic uses the following logistic regression model for the
logarithm of the odds ratio:
\begin{equation}
\label{eq:logistic-simple}
\log \frac{\hat p_n}{1-\hat p_n}= \log\frac{p_n}{1-p_n} + \thetat c_n,
\end{equation}
where $\theta \in{\mathbb R}^d$ is a parameter vector.

In previous studies in game-theoretic probability, many strategies of Skeptic depend only
on $x_i-p_i$, $i\le n-1$, and do not depend on $p_n$.  However obviously it is more reasonable
to consider Skeptic's strategies which depend on $p_n$. Strategies explicitly depending on $p_n$ are
also important from the viewpoint of defensive forecasting (\cite{defensive-forecasting},
\cite{df-linear}).  We again discuss this point in Section \ref{subsec:monotonicity}.

We now consider the capital process $\cK_n^\theta$ of \eqref{eq:logistic-simple} 
for a fixed $\theta\in \mathbb{R}^d$.
Solving for $\hat p_n$ we have
\begin{equation}
\label{eq:hatpn}
\hat p_n = \frac{p_n e^{\thetat c_n}}{1+ p_n (e^{\thetat c_n}-1)}, \qquad
1-\hat p_n = \frac{1-p_n}{1+ p_n (e^{\thetat c_n}-1)}.
\end{equation}
Then 
\[
\hat p_n^{x_n} (1-\hat p_n)^{1-x_n}= p_n^{x_n} (1- p_n)^{1-x_n} \frac{e^{\thetat c_n x_n}}{1+ p_n (e^{\thetat c_n}-1)}
\]
and the capital process is written as
\begin{equation}
\label{eq:capital-theta}
\cK_{n}^\theta = \prod_{i=1}^n \frac{\hat p_i^{x_i} (1-\hat p_i)^{1-x_i}}{p_i^{x_i} (1-p_i)^{1-x_i}}=
\frac{e^{\thetat \sum_{i=1}^n c_i x_i}}{\prod_{i=1}^n (1+p_i (e^{\thetat c_i}-1))}.
\end{equation}

Naturally it is better for Skeptic to choose the value of $\theta$ depending on the moves of other players. 
In this paper we consider a Bayesian strategy,
which specifies a prior distribution $\pi(\theta)$ for
$\theta$.  Bayesian strategies for Binary Probability Forecasting with constant $p_n\equiv p$ 
was considered in 
\cite{ktt-2008saa}. Bayesian strategy is basically the same as the universal portfolio by
Cover and his coworkers (\cite{cover-universal-portfolios}, \cite{cover-ordentlich}, \cite{cover-thomas-2nd}).
In the universal portfolio, a prior is put on the betting ratio $\nu$  itself, where as 
we put a prior on the parameter of Skeptic's model.  Furthermore differently from \cite{cover-ordentlich}
we allow continuous side information.  

In the Bayesian logistic strategy with the prior density function $\pi(\theta)$ of $\theta$, the
capital process $\cK_{n}^\pi$ is written as
\[
\cK_{n}^\pi = %
\int_{{\mathbb R}^d} \cK_n^\theta \pi(\theta)d\theta=
\int_{{\mathbb R}^d} %
\frac{e^{\thetat \sum_{i=1}^n c_i x_i}}{\prod_{i=1}^n (1+p_i (e^{\thetat c_i}-1))}
\pi(\theta)d\theta.
\]
$\cK_{n}^\pi$ is of the form \eqref{eq:joint-probability-formulation} where
\[
\hat p(x_1,\dots,x_n)= \prod_{i=1}^n p_i^{x_i} (1-p_i)^{1-x_i} \int_{{\mathbb R}^d}
\frac{e^{\thetat \sum_{i=1}^n c_i x_i}}{\prod_{i=1}^n (1+ p_i (e^{\thetat c_i}-1))}\pi(\theta)d\theta.
\]

In this paper we consider a prior density which is positive in a 
neighborhood of the origin.  We call such $\pi$ ``a prior supporting a neighborhood of
the origin''.

\section{Properties of logistic betting strategy from the viewpoint of game-theoretic probability}
\label{sec:forcing}

In this section we prove game-theoretic properties of our Bayesian logistic strategy.

\subsection{Weak forcing of the usual form of the strong law of large numbers}
\label{subsec:usual-SLLN}

The first theoretical result on our logistic betting strategy is the following theorem.
\begin{theorem}
\label{thm:usual-SLLN}
In BPFSI, by a Bayesian logistic strategy with a prior supporting a neighborhood of the origin, Skeptic can weakly force
\[
E_1 \ \Rightarrow \ \lim_n \cV_n^{-1}\cS_n =0,
\]
where $E_1$ is given in \eqref{eq:our-regularity}.
\end{theorem}

The rest  of this subsection is devoted to a proof of this theorem.
The basic logic of our proof is the same as in 
Section 3.2 of \cite{ShaVov01}.

We first consider  the logarithm of $\cK_n^\theta$ in \eqref{eq:capital-theta} for a fixed $\theta$: 
\[
\log \cK_n^\theta = \thetat \sum_{i=1}^n c_i x_i  - \sum_{i=1}^n \log(1+p_i (e^{\thetat c_i}-1)).
\]
For notational simplicity we write 
\[
\ckn(\theta)=\log \cK_n^\theta.
\]
We investigate the behavior of $\ckn(\theta)$ for 
$\theta$ close the origin.
Fix $\theta \in {\mathbb R}^d$ with unit length (i.e.\ $\Vert \theta\Vert=1$) and consider
$\ckn(s\theta)$, $0\le s\le \epsilon$.  Note that $\ckn(0)=0$.
We will choose $\epsilon$ appropriately later
in \eqref{eq:epsilon-selec}.

The derivative of $\ckn(s\theta)$ with respect to $s$
is written as follows.
\begin{align}
\frac{\partial}{\partial s} \ckn(s\theta) &=  \thetat \sum_{i=1}^n c_i x_i  
- \sum_{i=1}^n \frac{\thetat c_i p_i e^{s\thetat c_i}}{1+p_i (e^{s\thetat c_i}-1)}
\nonumber \\
&=\thetat \sum_{i=1}^n c_i (x_i - p_i) - \sum_{i=1}^n \frac{\thetat c_i p_i e^{s\thetat c_i} - \thetat c_i p_i (1+p_i (e^{s\thetat c_i}-1))
}{1+p_i (e^{s\thetat c_i}-1)}\nonumber \\
&=\thetat \cS_n  - \sum_{i=1}^n \thetat c_i p_i (1-p_i) 
\frac{e^{s\thetat c_i}-1} {1+p_i (e^{s\thetat c_i}-1)} .
\label{eq:fixed-theta-capital}
\end{align}
Note that $\thetat c_i$ and $e^{s\thetat c_i}-1$ have the same sign and hence
each summand in the second term on the right-hand side of \eqref{eq:fixed-theta-capital} is non-negative.

Let 
\[
\gamma_p(y)= \frac{e^y-1}{1+p(e^y-1)}  
\]
be a function of $y\in {\mathbb R}$ depending on the parameter $p\in [0,1]$.  Note $\gamma_p(0)=0$.  
Its derivative is computed as 
\begin{equation}
\label{eq:gamma-deriv}
\gamma_p'(y)= \frac{e^y}{(1+p(e^y-1))^2} > 0.
\end{equation}
Hence 
\[
\gamma_p(y)=\int_0^y \gamma_p'(z)dz = \int_0^y  \frac{e^z}{(1+p(e^z-1))^2} dz,
\]
where for negative $y<0$  we interpret $\int_0^y (\cdots) dz $ as $-\int_y^0 (\cdots) dz$.
Now $\gamma_p'(z)$ in \eqref{eq:gamma-deriv}
is monotone in $p$ with $\gamma'_0(z)=e^z$ and 
$\gamma_1'(z)=e^{-z}$.
Hence 
\[
e^{-|z|} = \min(e^{-z},e^z) \le \gamma_p'(z) \le \max(e^{-z},e^{z})=e^{|z|}.
\]
Then for $z$ between $0$ and $y$ we have 
\begin{equation}
\label{eq:second-deriv-2}
e^{-|y|} \le \gamma_p'(z) \le e^{|y|}.
\end{equation}
Using the upper bound $e^{|y|}$ and integrating $\gamma_p'(z)$ we obtain
\[
|\gamma_p(y)|=\frac{|e^y-1|}{1+p(e^y-1)} \le |y| e^{|y|} \quad\text{and}\quad 0 \le y\gamma_p(y) = y \frac{e^y-1}{1+p(e^y-1)} \le y^2 e^{|y|}.
\]

Let  $L_{c,n}=\max_{1\le i \le n} \Vert c_i\Vert$.
Then 
\[
\frac{\partial}{\partial s} \ckn(s\theta) \ge
\thetat \cS_n  - s \sum_{i=1}^n (\thetat c_i)^2 p_i (1-p_i) 
e^{\epsilon L_{c,n}} = \thetat \cS_n  - s \thetat \cV_n \theta e^{\epsilon L_{c,n}}
\]
and integrating this for $0\le s\le \epsilon$ we have (for any $\theta$ and $\epsilon>0$)
\[
\ckn(\epsilon \theta)\ge \epsilon \thetat \cS_n - \frac{\epsilon^2}{2} \thetat \cV_n \theta e^{\epsilon L_{c,n}}.
\]

For the rest of our proof we arbitrary choose and fix a path $\xi\in E_1$, where $E_1$ is given in 
\eqref{eq:our-regularity}.  Various constants ($\epsilon$'s,  $L$'s etc.) below may depend on $\xi$.
By iii) there exists $L_c$ such that
$L_{c,n} < L_c$ for all $n$. Also there exist $n_0$ and $L_\lambda$ such that
$\lambda_{\max,n}/\lambda_{\min,n}< L_\lambda$ for all $n\ge n_0$.
Now suppose that  $\cV_n^{-1} \cS_n \not \rightarrow 0$ for this $\xi$. 
Then for some  $\epsilon_1 > 0$  and for infinitely many $n$  
we have $\Vert \cV_n^{-1} \cS_n\Vert \ge \epsilon_1$.  Let 
${\mathbb N}_1 = \{ n_1, n_2, \dots\}$ be a subsequence such that 
$\Vert \cV_n^{-1} \cS_n\Vert \ge \epsilon_1$ for $n\in {\mathbb N}_1$.
The normalized vectors
\[
\eta_n = \frac{\cV_n^{-1} \cS_n}{\Vert \cV_n^{-1} \cS_n \Vert}, \ \ n\in {\mathbb N}_1,
\]
have an accumulation point $\eta$, $\Vert \eta\Vert=1$, 
and hence along a further subsequence ${\mathbb N}_2\subset {\mathbb N}_1$ we have
\[
\lim_{n\rightarrow\infty,\  n \in {\mathbb N}_2} \eta_n = \eta.
\]

By Cauchy-Schwarz, for three vectors $a,b,c\in {\mathbb R}^d$,  we have
\[
\frac{|b\trns \cV_n c|}{a\trns \cV_n a} 
\le \frac{\lambda_{\max,n} \Vert b\Vert \Vert c \Vert}{\lambda_{\min,n} \Vert a\Vert^2}
< L_\lambda \frac{ \Vert b\Vert \Vert c \Vert}{ \Vert a\Vert^2}, \quad \forall n\ge n_0.
\]
Then we can choose $0< \epsilon_2< 1/4$ such that 
for all sufficiently large $n\in {\mathbb N}_2$ and
for all $\tilde \eta$, $\Vert\tilde\eta\Vert=1$, sufficiently close to $\eta$, we have
\begin{align*}
\ckn(\epsilon \tilde \eta) 
&\ge \epsilon \tilde\eta\trns \cS_n - \frac{\epsilon^2}{2} \tilde \eta\trns
\cV_n \tilde \eta e^{\epsilon L_{c,n}} \\
&= \epsilon \Vert \cV_n^{-1} \cS_n \Vert
\tilde\eta\trns \cV_n \eta_n - \frac{\epsilon^2}{2} \tilde \eta\trns
\cV_n \tilde \eta e^{\epsilon L_{c,n}} 
\\
&\ge  \epsilon \epsilon_1 \eta\trns \cV_n \eta(1-\epsilon_2) -
\frac{\epsilon^2}{2} \eta\trns \cV_n \eta(1+\epsilon_2) e^{\epsilon L_c} \\
&=\epsilon \eta\trns \cV_n \eta \; \big(\epsilon_1(1-\epsilon_2) - \frac{\epsilon}{2}(1+\epsilon_2)e^{\epsilon L_c}\big).
\end{align*}

We now choose small enough  $\epsilon>0$  such that
\begin{equation}
\label{eq:epsilon-selec}
\epsilon_1 (1-\epsilon_2) - \frac{\epsilon}{2}(1+\epsilon_2)e^{\epsilon L_c}> \frac{\epsilon_1}{2}.
\end{equation}
Then
\[
\ckn(\epsilon \tilde \eta) \ge \frac{\epsilon \epsilon_1}{2} \lambda_{\min,n}   \rightarrow \infty  \qquad (n\rightarrow\infty, n\in {\mathbb N}_2).
\]
Note that the convergence is uniform for $\tilde \eta$ in some neighborhood $N(\eta)$ of $\eta$.
Since our prior $\pi$ puts a positive weight to $N(\epsilon\eta)$, 
$\cK_n^\pi \rightarrow\infty$ along $n\in {\mathbb N}_2$.  
This completes our proof of Theorem \ref{thm:usual-SLLN}.

\subsection{Weak forcing of a more precise form of the strong law of large numbers}
\label{subsec:main-result}
As discussed in Section 
\ref{sec:formulation}, we can establish a much more precise rate of convergence
of the strong law of large numbers based on our Bayesian logistic strategy.
Our main theorem of this paper is stated as follows.
\begin{theorem}
\label{thm:main}
In BPFSI, by a Bayesian logistic strategy with a prior distribution 
supporting a neighborhood of the origin, Skeptic can weakly force
\[
E_1 \ \Rightarrow \ \limsup_n \frac{\cS_n\trns \cV_n^{-1}\cS_n}{\log \det \cV_n} \le 1,
\]
where $E_1$ is given in \eqref{eq:our-regularity}.
\end{theorem}
We give a proof of this theorem in the following three subsections.

\subsubsection{Bounding the maximum likelihood estimate}
We now consider the behavior of $\cK_n^\theta$ in
\eqref{eq:capital-theta}, when $\cK_n^\theta$ is maximized
with respect to $\theta$.  Let
\[
\mle  =\operatorname{argmax} \cK_n^\theta.
\]
We call $\mle$ the maximum likelihood estimate, since $\cK_n^\theta$ is of the form of the likelihood function of the logistic regression model.
It is easily seen that the maximizer $\mle$ is finite except for a
special case that the vectors in $\{ c_i \mid x_i=1\}\cup \{ -c_i \mid x_i=0\}$
lie on a half-space defined by a hyperplane containing the origin.  More specifically in Lemma
\ref{lem:small-mle}  we prove that $\Vert \mle\Vert $ is small when $\Vert \cV_n^{-1}\cS_n\Vert$
is small.

The maximizing $\mle$ can only be computed at the end of day $n$
after seeing all the data $p_1, c_1, x_1, \dots, p_n, c_n, x_n$.  
Hence we call a strategy using $\mle$ a ``hindsight
strategy'', which is the same as the best constant rebalanced portfolio (BCRP) in the terminology
of the universal portfolio.

We prove the following lemma.
\begin{lemma}
\label{lem:small-mle}
Let $L_{c,n} = \max_{1\le i \le n} \Vert c_i \Vert$
and $L_{\lambda,n}=\lambda_{\max,n}/\lambda_{\min,n}$,
where we assume $\lambda_{\min,n}>0$. Then
\[
\Vert \cV_n^{-1} \cS_n \Vert \le \frac{1}{3L_{c,n} L_{\lambda,n}} \ \ \Rightarrow \ \ \Vert\mle \Vert
\le 3 L_{\lambda,n} \Vert \cV_n^{-1} \cS_n \Vert.
\]
\end{lemma}

For any fixed $\xi\in E_1$, there exist $L_c, L_\lambda$, 
such that $L_{c,n}< L_c$ and $L_{\lambda,n}<L_\lambda$ 
for all sufficiently large $n$.  Also in Theorem \ref{thm:usual-SLLN} we proved
that Skeptic can weakly force $E_1 \Rightarrow \lim_n \cV_n^{-1}\cS_n=0$.  From these
results we have the following proposition.

\begin{proposition}
\label{prop:mle0}  
In the same setting as in Theorem
\ref{thm:usual-SLLN} Skeptic can weakly force
$
E_1 \ \Rightarrow \ \lim_n \mle=0
$.
\end{proposition}

The rest  of this subsection is devoted to a proof of Lemma \ref{lem:small-mle}.
Consider the inner product $\theta\trns \nabla \ckn(\theta)=\theta\trns \grad \ckn(\theta)$
of $\theta$ and the gradient of $u(\theta)$. 
If $\theta\trns \nabla \ckn(\theta)\le 0$, then the  gradient points toward the interior 
of the ball with radius $r=\Vert\theta\Vert$ as shown in Figure \ref{fig:1}.
If $\theta\trns \nabla \ckn(\theta)\le 0$ for all $\theta$ with $\Vert\theta\Vert=r$, then
$\Vert \mle\Vert \le r$.
\begin{figure}[htbp]
\begin{center}
\includegraphics[width=4cm]{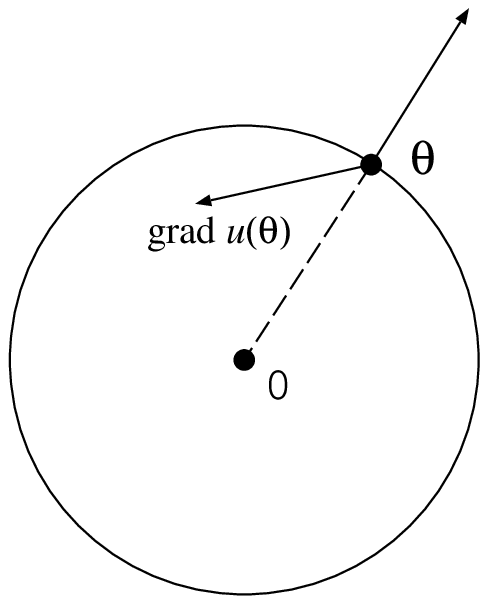}
\caption{Gradient of $\ckn(\theta)$}
\label{fig:1}
\end{center}
\end{figure}
This can be seen as follows.  Suppose $\Vert \mle\Vert>r$.
Let $\tilde\theta$ be the maximizer of $\ckn(\theta)$ on the
sphere (the boundary of the ball).  Then at $\tilde\theta$ the 
gradient of $\nabla \ckn(\tilde \theta)$ is a 
positive  multiple of $\theta$ and this contradicts 
$\tilde \theta\trns \nabla \ckn(\tilde \theta)\le 0$.

As in the previous subsection, using this time 
the lower bound in \eqref{eq:second-deriv-2}, we have
\[
\thetat \nabla u(\theta) \le  
\theta\trns \cS_n - \thetat\cV_n \theta e^{-L_{c,n}\Vert \theta\Vert}.
\]
Now 
\[
|\thetat\cS_n| =|\thetat \cV_n \cV_n^{-1} \cS_n| \le \Vert \thetat \cV_n \Vert \cdot
\Vert  \cV_n^{-1} \cS_n \Vert 
\]
and 
\[
\Vert \thetat \cV_n \Vert^2 =\thetat \cV_n^2 \theta \le 
\Vert \theta\Vert^2 \lambda_{\max,n}^2.
\]
Hence
\[
|\thetat\cS_n| \le  \Vert \theta\Vert \lambda_{\max,n}\Vert  \cV_n^{-1} \cS_n \Vert .
\]
Furthermore
\[
\thetat\cV_n \theta e^{-L_{c,n}\Vert \theta\Vert} \ge 
\lambda_{\min,n} \Vert\theta\Vert^2  e^{-L_{c,n}\Vert \theta\Vert} .
\]
Therefore
\[
\thetat \nabla u(\theta) %
\le
\lambda_{\min,n}\Vert \theta\Vert \big( L_{\lambda,n} \Vert  \cV_n^{-1} \cS_n \Vert - \Vert\theta\Vert e^{-L_{c,n}\Vert \theta\Vert} \big).
\]
For $\Vert\theta\Vert=  3 L_{\lambda,n} \Vert  \cV_n^{-1} \cS_n \Vert$
\[
L_{\lambda,n} \Vert  \cV_n^{-1} \cS_n \Vert - \Vert\theta\Vert e^{-L_{c,n}\Vert \theta\Vert} 
= L_{\lambda,n} \Vert  \cV_n^{-1} \cS_n \Vert ( 1-  3e^{-3  L_{c,n} L_{\lambda,n} \Vert\cV_n^{-1} \cS_n\Vert})
\]
Then for $\Vert \cV_n^{-1} \cS_n \Vert \le 1/(3L_{c,n} L_{\lambda,n})$
\[
3e^{-3 L_{\lambda,n} L_{c,n}\Vert\cV_n^{-1} \cS_n\Vert}\ge 3 e^{-1}>1.
\]
Hence, if $\Vert \cV_n^{-1} \cS_n \Vert \le 1/(3L_{c,n} L_{\lambda,n})$,  
we have $\thetat \nabla\ckn(\theta) < 0$ 
for all $\theta$ with $\Vert\theta\Vert=  3 L_{\lambda,n} \Vert  \cV_n^{-1} \cS_n \Vert$.  
By the remark just after Proposition
\ref{prop:mle0}, this completes the proof of 
Lemma \ref{lem:small-mle}.

\subsubsection{Behavior of the hindsight strategy}
\label{sec:hindsight}

We summarize properties of $\log\cK_n^{\mle}$ in view of the standard theory of exponential families
(\cite{barndorff-nielsen}) in statistical inference.
Define 
\[
\psi_i(\theta)=\log(1+p_i (e^{\theta c_i}-1)),
\quad
\psi(\theta)= \sum_{i=1}^n \psi_i(\theta).
\]
Note that $\psi_i(\theta)$ is 
the cumulant generating function (potential function) for the logistic regression model,
which is an exponential family model with the natural parameter $\theta$.
Hence  each $\psi_i(\theta)$ and  $\psi(\theta)$ are convex in $\theta$. 
Indeed  by \eqref{eq:gamma-deriv}, the Hessian matrix $H_{\psi_i}(\theta)$ of $\psi_i$ is given as
\[
H_{\psi_i}(\theta)=
c_i c_i\trns \frac{p_i (1-p_i)e^{\thetat c_i}}{(1+ p_i (e^{\thetat c_i}-1))^2},
\]
which is non-negative definite.  The Hessian matrix 
\[
H_{\psi}(\theta)=\sum_{i=1}^n H_{\psi_i}(\theta)
\]
of $\psi$ is positive definite if $\cV_n$ is
positive definite, which is the Fisher information matrix 
in terms of the natural parameter $\theta$.

Convexity of $\psi_i$ implies concavity of  $\log \cK_n^\theta=\thetat \cT_n - \psi(\theta)$, where
\[
\cT_n = \sum_{i=1}^n c_i x_i = \cS_n + \sum_{i=1}^n c_i p_i.
\]
Hence if the maximum of 
$\log \cK_n^\theta$ is attained at  a finite value $\mle$, then the ``maximum likelihood estimate''
$\mle$ satisfies %
``the likelihood equation'' 
\[
\frac{\partial}{\partial \theta} \log \cK_n^\theta=0
\]
or equivalently
\begin{equation}
\label{eq:likelihood-equation2}
\cT_n = \nabla \psi(\mle).
\end{equation}
The likelihood equation can also be written as
\[
0=\sum_{i=1}^n (x_i - \hat p_i^*)c_i,  \qquad
\hat p_i^*  = \hat p_{i;n}^*
= \frac{p_i e^{\mle c_i}}{1+p_i (e^{\mle c_i}-1)}.
\]
From this it follows that $\mle=0$ if and only if $\cT_n = \sum_{i=1}^n c_i p_i$.

Regard \eqref{eq:likelihood-equation2} as determining $\mle$ in terms of $t=\cT_n$, i.e., 
$\mle = \mle(t), t=\cT_n$.  This is the inverse map of $t=\nabla \psi(\theta)$.
Differentiating $t=\nabla \psi(\theta)$ again with respect to $\theta$ we obtain the Jacobi matrix
\[
J=\frac{\partial t}{\partial \theta}=H_\psi(\theta)
\]
as the Hessian matrix of $\psi$. Hence the Jacobi matrix $\partial \mle/\partial \cT_n$
is written as
\begin{equation}
\label{eq:hessian-inv}
\frac{\partial \mle}{\partial \cT_n}=H_\psi(\mle(\cT_n))^{-1}.
\end{equation}

Now $\log\cK_n^{\mle}=\log\cK_n^{\mle(\cT_n)}$ is the Legendre transformation (cf.\ Chapter 3 of
\cite{arnold1989}) of $\log \cK_n^\theta$:
\[
\log\cK_n^{\mle(t)} = \mle(t)\trns t - \psi(\mle(t)), \qquad t=\cT_n.
\]
Differentiating $\log\cK_n^{\mle(t)}$ with respect to $t$, by \eqref{eq:likelihood-equation2} we obtain
\begin{equation}
\label{eq:mle-t}
\frac{\partial}{\partial t}\log\cK_n^{\mle(t)} = \mle(t) + (t - \nabla \psi(\mle(t))\frac{\partial\mle}{\partial t}
=\mle(t).
\end{equation}
By \eqref{eq:hessian-inv} the Hessian matrix of $\log\cK_n^{\mle(t)}$ 
is given by $H_\psi(\mle(t))^{-1}$.

We are now ready to prove the following proposition.
\begin{proposition}
\label{prop:small-mle}
With the same setting as in Lemma \ref{lem:small-mle}, 
\[
\Vert \cV_n^{-1} \cS_n \Vert \le \frac{1}{3L_{c,n} L_{\lambda,n}} \ \ \Rightarrow\ \ 
e^{-C_n \Vert \cV_n^{-1} \cS_n \Vert} \le \frac{\log \cK_n^{\mle}}{\cS_n\trns \cV_n \cS_n/2} \le 
e^{C_n \Vert \cV_n^{-1} \cS_n \Vert},
\]
where $C_n = 3  L_{c,n} L_{\lambda,n}$.
\end{proposition}

\begin{proof}
For  given $\cT_n$, $\bar \cT_0 = \sum_{i=1}^n c_i p_i$  and for $s\in [0,1]$, consider 
\[
g(s)= \mle(\bar \cT_0 + s \cS_n)\trns (\bar \cT_0 + s \cS_n) - \psi(\mle(\bar \cT_0 + s\cS_n)).
\]
Then $\log \cK_n^{\mle(\cT_n)} = g(1)$.  It is easily seen that  $g(0)=0$. By \eqref{eq:mle-t}
\[
g'(s)= \mle(\bar \cT_0 + s \cS_n)\trns \cS_n.
\]
Again it is easily seen that $g'(0)=0$, since $\mle(\bar \cT_0)=0$.  Then
\[
g(1)=\int_0^1 \int_0^s g''(u)du  ds.
\]
Now
\[
g''(u)=\cS_n\trns H_\psi(\mle(\bar \cT_0 + u\cS_n))^{-1} \cS_n.
\]
By \eqref{eq:second-deriv-2}
\begin{align*}
e^{-\Vert \mle(\bar \cT_0 + u\cS_n) \Vert L_{c,n}} \cS_n\trns \cV_n^{-1} \cS_n 
&\le \cS_n\trns H_\psi(\mle(\bar\cT_0 + u\cS_n))^{-1} \cS_n \\
& \le e^{\Vert \mle(\bar\cT_0 + u\cS_n) \Vert L_{c,n}} \cS_n\trns \cV_n^{-1}  \cS_n. 
\end{align*}
Also $\int_0^1 \int_0^s 1 du  ds = 1/2$.  
Furthermore by Lemma \ref{lem:small-mle}, if 
$\Vert \cV_n^{-1} \cS_n \Vert \le 1/(3L_{c,n} L_{\lambda,n})$  then
$\Vert\mle(\bar\cT_0 + u\cS_n) \Vert \le 3 L_{\lambda,n} \Vert \cV_n^{-1} \cS_n \Vert$
for all $0\le u\le 1$.
Combining these results we have the proposition.
\end{proof}

As in Proposition \ref{prop:small-mle} we have the following corollary.
\begin{corollary}
\label{cor:hindsight}
In the same setting as in Theorem
\ref{thm:usual-SLLN} Skeptic can weakly force
\[
E_1 \ \Rightarrow \ 
\lim_n \frac{\log \cK_n^{\mle}}{\cS_n\trns \cV_n^{-1} \cS_n/2}  = 1.
\]
\end{corollary}

\subsubsection{Laplace method for evaluating 
the difference of the hindsight strategy and the logistic strategy}
\label{sec:diff-hindsight}
In the last subsection we clarified the behavior of the capital process
for the hindsight strategy.  Now we employ the standard Laplace method to
evaluate the difference of the hindsight strategy and the logistic strategy
(Section 5 of \cite{cover-universal-portfolios}, Chapter 3.1 of \cite{jensen}).

\begin{lemma}
\label{lem:laplace}
Let $\pi$ be a prior density supporting a neighborhood of the origin and let $\cK_n^\pi$ 
denote its capital process.  For  $\xi \in E_1$ such that $\lim_n\cV_n^{-1}\cS_n=0$,
\begin{equation}
\label{eq:laplace}
\lim_n \frac{ \log \cK_n^{\mle} - \log  \cK_n^\pi }{(1/2)\log \det \cV_n} = 0.
\end{equation}
\end{lemma}
\begin{proof}
For $\theta$ close to the origin, 
expanding $\log\cK_n^\theta$ around $\mle$ we have
\[
\log\cK_n^\theta = \log \cK_n^{\mle} - \frac{1}{2} (\theta - \mle)\trns H_\psi(\tilde \theta_n)(\theta - \mle),
\]
where $\tilde \theta_n$ is on the line segment joining $\theta$ and $\mle$.  Hence
\[
\cK_n^\theta = \cK_n^{\mle} \times \exp(-\frac{1}{2} (\theta - \mle)\trns H_\psi(\tilde \theta_n)(\theta - \mle)).
\]
Now by the standard Laplace method we obtain \eqref{eq:laplace}.
\end{proof}

Finally we give a proof of Theorem \ref{thm:main}.
\begin{proof}[Proof of Theorem \ref{thm:main}]
By Corollary \ref{cor:hindsight} and Lemma \ref{lem:laplace}
\[
\log \cK_n^\pi = \frac{1}{2} \log\det \cV_n \; 
\big( \frac{ \cS_n\trns \cV_n^{-1}\cS_n} {\log \det \cV_n} - 1 + o(1)\big).
\]
Hence if $\limsup_n \cS_n\trns \cV_n^{-1}\cS_n/ \log \det \cV_n> 1$, then
$\limsup_n \log \cK_n^\pi = \infty$.
\end{proof}

\subsection{Monotonicity with respect to the forecast probability}
\label{subsec:monotonicity}

Here we consider the case that $\log (p_n/(1-p_n))$ itself is an element of the vector of the side information  $c_n$
and hence is multiplied by a coefficient in \eqref{eq:logistic-simple}.
For notational convenience we here eliminate $\log (p_n/(1-p_n))$ from $c_n$ and write 
\eqref{eq:logistic-simple} as
\begin{equation}
\label{eq:logistic-monotone}
\log \frac{\hat p_n}{1-\hat p_n}= \beta \log\frac{p_n}{1-p_n} + \tau_n,
\end{equation}
where $\tau_n$ denotes the effect of side information other than 
$\log (p_n/(1-p_n))$.  Intuitively $\beta$ represents how much trust Skeptic puts in Forecaster.
If $\beta=0$ then  Skeptic entirely ignores Forecaster's $p_n$ and if $\beta=1$ then Skeptic
takes $p_n$ for granted.  The value of $\beta\in (0,1)$ corresponds to partial trust in $p_n$.
It is somewhat surprising to see that $\beta > 1$ 
in the case of probability of precipitation announced by 
the Japan Meteorological Agency in  Section \ref{subsec:JMA}.

We now investigate  how $\nu_n$ in \eqref{eq:nu_n} behaves with respect to $p_n$ for 
given $p_1, c_1, x_1, \dots, p_{n-1},c_{n-1},\allowbreak x_{n-1}$.
This is an important question from the viewpoint of defensive forecasting (\cite{defensive-forecasting},
\cite{df-linear}), because in defensive forecasting we want to obtain $p_n$ for which $\nu_n=0$.
For notational simplicity we now omit the subscript $n$ and write 
\eqref{eq:hatpn} as
\[
\hat p = \frac{ p \left(\frac{p}{1-p}\right)^{\beta-1} e^\tau}{1+p \left(\frac{p}{1-p}\right)^{\beta-1} e^\tau}.
\]
Then 
\[
\nu(p)=\frac{\hat p- p}{p(1-p)} = \frac{p^{\beta-1} e^\tau - (1-p)^{\beta-1}}{p^\beta e^\tau + (1-p)^\beta}.
\]
Differentiating this with respect to $p$ we obtain
\[
\frac{d\nu(p)}{dp}=\frac{ - e^{2\tau} p^{2(\beta-1)}  + e^\tau p^{\beta-2}(1-p)^{\beta-2}(\beta-2 + 2p(1-p)) - (1-p)^{2\beta-2}}
{(p^\beta e^\tau + (1-p)^\beta)^2}.
\]
The numerator of $d\nu(p)/dp$ can be written as
\[
-(e^\tau p^{\beta-1} - (1-p)^{\beta -1})^2 + e^\tau (\beta-1) p^{\beta-2} (1-p)^{\beta-2},
\]
which is non-positive for $\beta\le 1$. 
Hence  we have the following proposition.
\begin{proposition}
\label{prop:monotonicity}
Under the logistic regression model \eqref{eq:logistic-monotone}, for $\beta\le 1$ 
the betting ratio $\nu_n(p_n)$ is monotone decreasing in $p_n$.
\end{proposition}

It is natural that $\nu_n$ is monotone decreasing in $p_n$, because if $p_n$ is too high and Skeptic
does not believe it, then Skeptic will bet on the non-occurrence  $x_n=0$.

For the special case of $\beta=1$, 
\[
\nu_n(p_n)=\frac{e^{\tau_n}-1}{1 + p_n (e^{\tau_n}-1)},
\]
which is bounded and monotone in $p_n \in [0,1]$.  For $\beta<1$,  $\nu_n(p_n)$ is unbounded and
it can be easily seen that
\[
\lim_{p_n \downarrow 0} \frac{\nu_n(p_n)}{1/p_n}=1, \qquad
\lim_{p_n \uparrow 1} \frac{\nu_n(p_n)}{1/(1-p_n)}=-1.
\]
We can interpret  the first limit as follows.  Suppose that $p_n=1/1000$, i.e.\ the price of a ticket is 1/1000 of a dollar.  
In this case Skeptic can buy $1000$ tickets with one dollar and has the chance of winning $1000$ dollars.
Hence Skeptic may want to buy $1000$ tickets.  Thus it is reasonable that  $\nu$ and $p_n$ are inversely proportional
when $p_n$ is small.

\section{Experiments}

\label{sec:experiments}

In this section we give some numerical studies of our strategy.  In Section 
\ref{subsec:simulation} we present some simulation results and in Section
\ref{subsec:JMA} we apply our strategy against 
probability forecasting  by the Japan Meteorological Agency.

\subsection{Some simulation studies}
\label{subsec:simulation}

We consider three cases and apply three strategies to these examples.  In our simulation studies
Reality chooses her moves probabilistically, either by Bernoulli trials or by a Markov chain model.
\begin{itemize}
\item Case 1:\;$x_n$ is a Bernoulli variable with the success probability 0.7 and
$p_n$ alternates between $0.4$ and $0.6$ (i.e.\ $0.4=p_1= p_3 = \cdots$ and  $0.6=p_2 = p_4 = \cdots$).
\item Case 2:\;$x_n$ is a Bernoulli variable with the success probability 0.5 and
$p_n$ alternates between $0.4$ and $0.6$.
\item Case 3:\;$p_n=0.5$ and $x_n$ is generated by a Markov chain model with transition probabilities shown in Figure \ref{fig:2}.
\end{itemize}

\begin{figure}[htbp]
\begin{center}
\includegraphics[width=4.5cm]{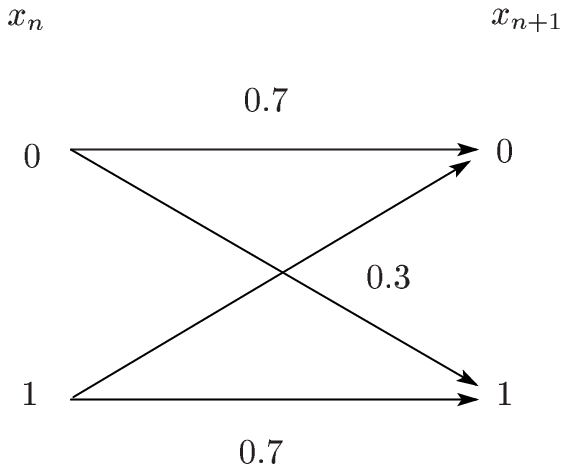}
\caption{Transition probabilities for  $x_n$}
\label{fig:2}
\end{center}
\end{figure}

\begin{itemize}
\item Strategy 1:\;$\theta$ is a scalar and $c_n=1$ in (\ref{eq:logistic-simple}). 
Assume that the prior density for $\theta$ is given as uniform distribution for [0,1].  The capital process is written as
\[
\cK_{n}^\pi = \int_0^1\frac{e^{\theta\sum_{i=1}^nx_i}}{\prod_{i=1}^n(1+p_i(e^\theta-1))}d\theta. 
\]

\item Strategy 2:\;$\thetat=[\theta_1,\beta-1]$ and $c_n\trns=[1,\log\frac{p_n}{1-p_n}]$. Assume independent priors for
$\theta_1$ and $\beta$, which are uniform distributions over [0,1].  The capital process is written as
\[
\cK_{n}^\pi = \int_0^1\int_0^1\frac{e^{\theta_1\sum_{i=1}^nx_i+(\beta-1)\sum_{i=1}^nx_i\log\frac{p_i}{1-p_i}}}
{\prod_{i=1}^n(1+p_i(e^{\theta_1+(\beta-1)\log\frac{p_i}{1-p_i}}-1))}d\theta_1d\beta  .
\]
\item Strategy 3:\;$\thetat=[\theta_1,\beta-1,\theta_3]$ and $c_n\trns=[1,\log\frac{p_n}{1-p_n},x_{n-1}]$. 
Assume independent priors for $\theta_1$, $\beta$ and $\theta_3$, which are uniform distributions over [0,1]. 
The capital process is written as
\[
\cK_{n}^\pi = \int_0^1\int_0^1\int_0^1\frac{e^{\theta_1\sum_{i=1}^nx_i+(\beta-1)\sum_{i=1}^nx_i\log\frac{p_i}{1-p_i}+\theta_3\sum_{i=1}^nx_ix_{i-1}}}{\prod_{i=1}^n(1+p_i(e^{\theta_1+(\beta-1)\log\frac{p_i}{1-p_i}+\theta_3x_{i-1}}-1))}d\theta_1d\beta d\theta_3 .
\]
\end{itemize}

\begin{figure}[htbp]
 \begin{minipage}{0.5\hsize}
   \includegraphics[width=85mm]{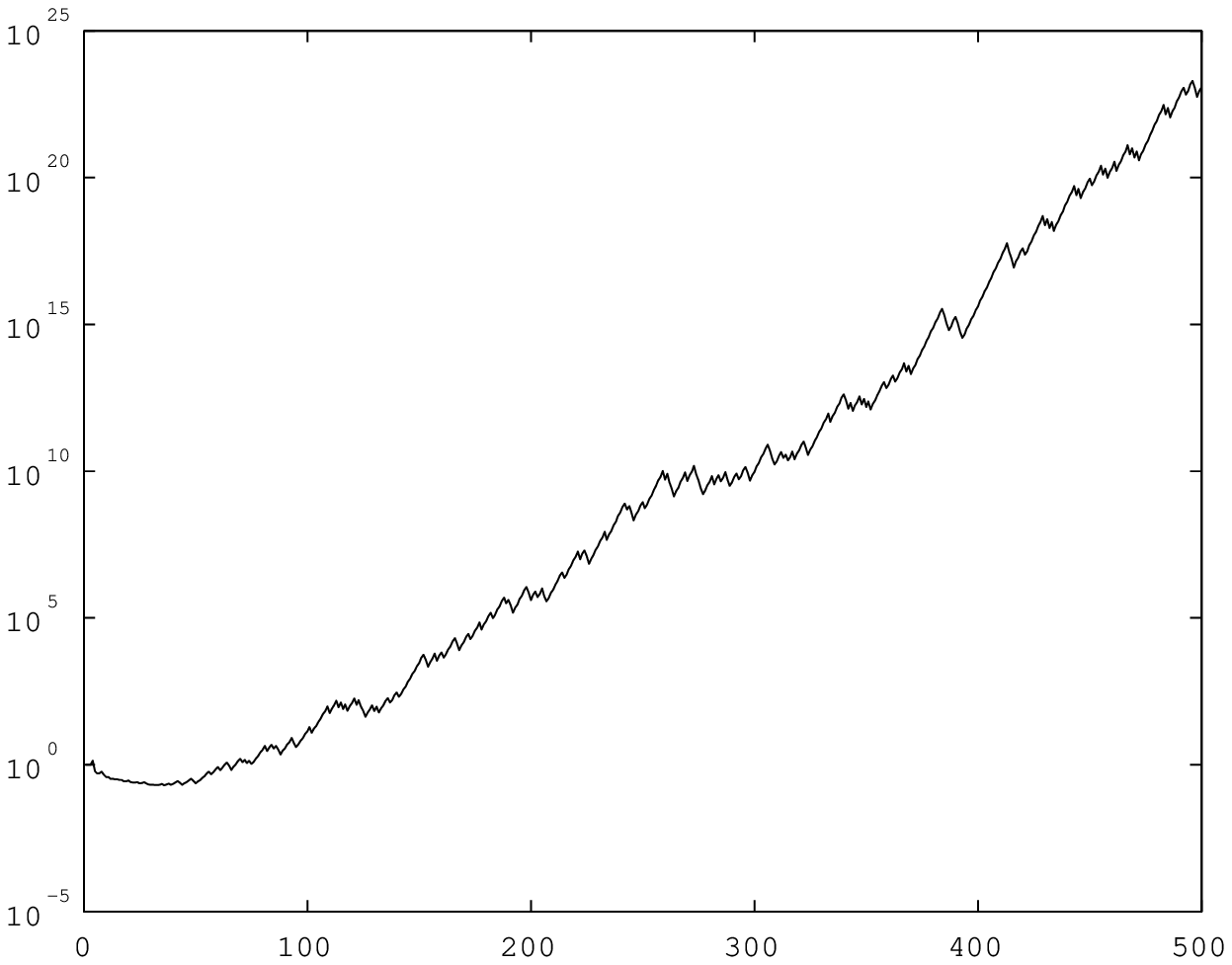}
  \caption{Case 1 with strategy 1}
  \label{st1_cs1}
 \end{minipage}
 \begin{minipage}{0.5\hsize}
   \includegraphics[width=85mm]{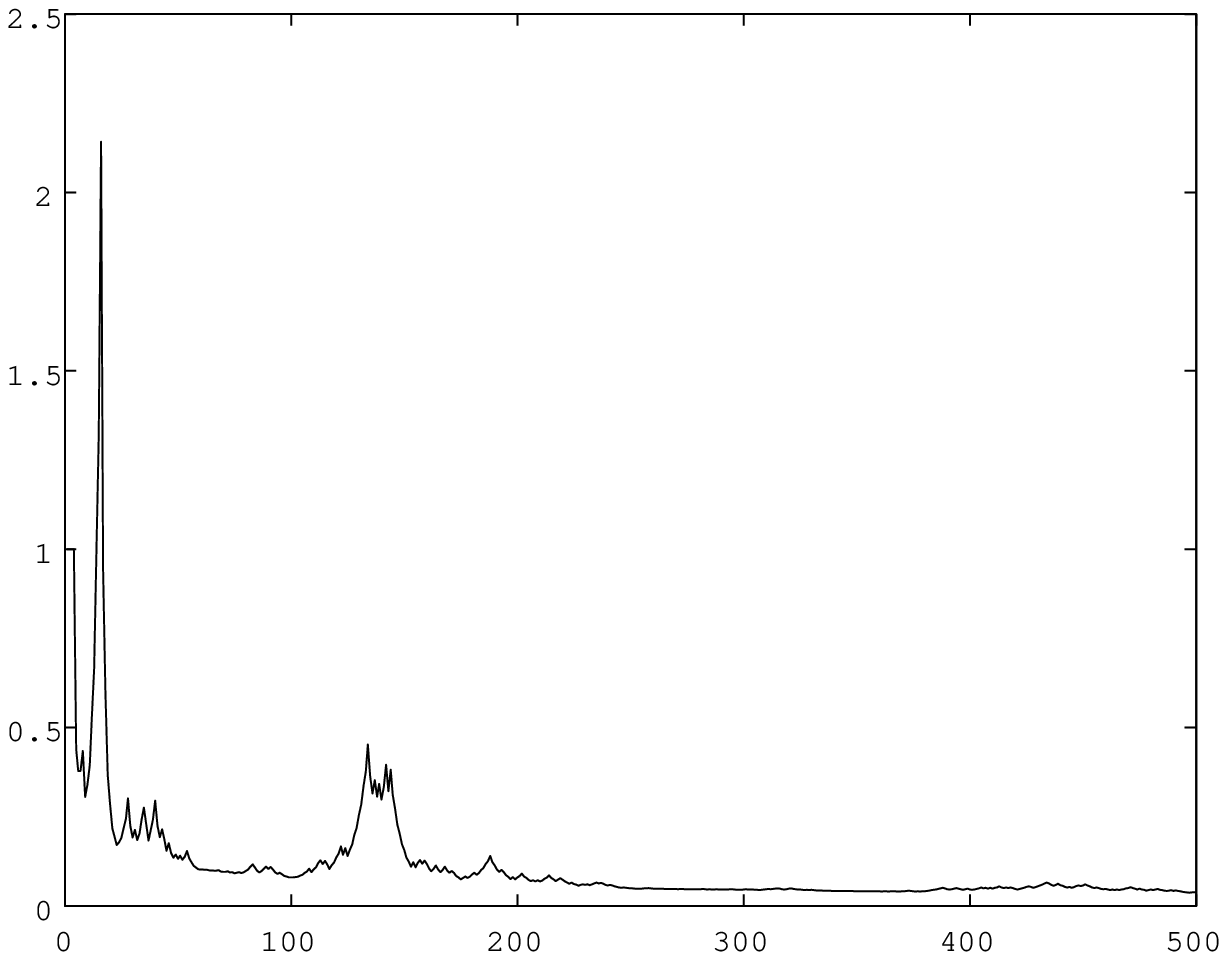}
  \caption{Case 2 with strategy 1}
  \label{st1_cs2}
 \end{minipage}
\end{figure}

\begin{figure}[htbp]
 \begin{minipage}{0.5\hsize}
   \includegraphics[width=85mm]{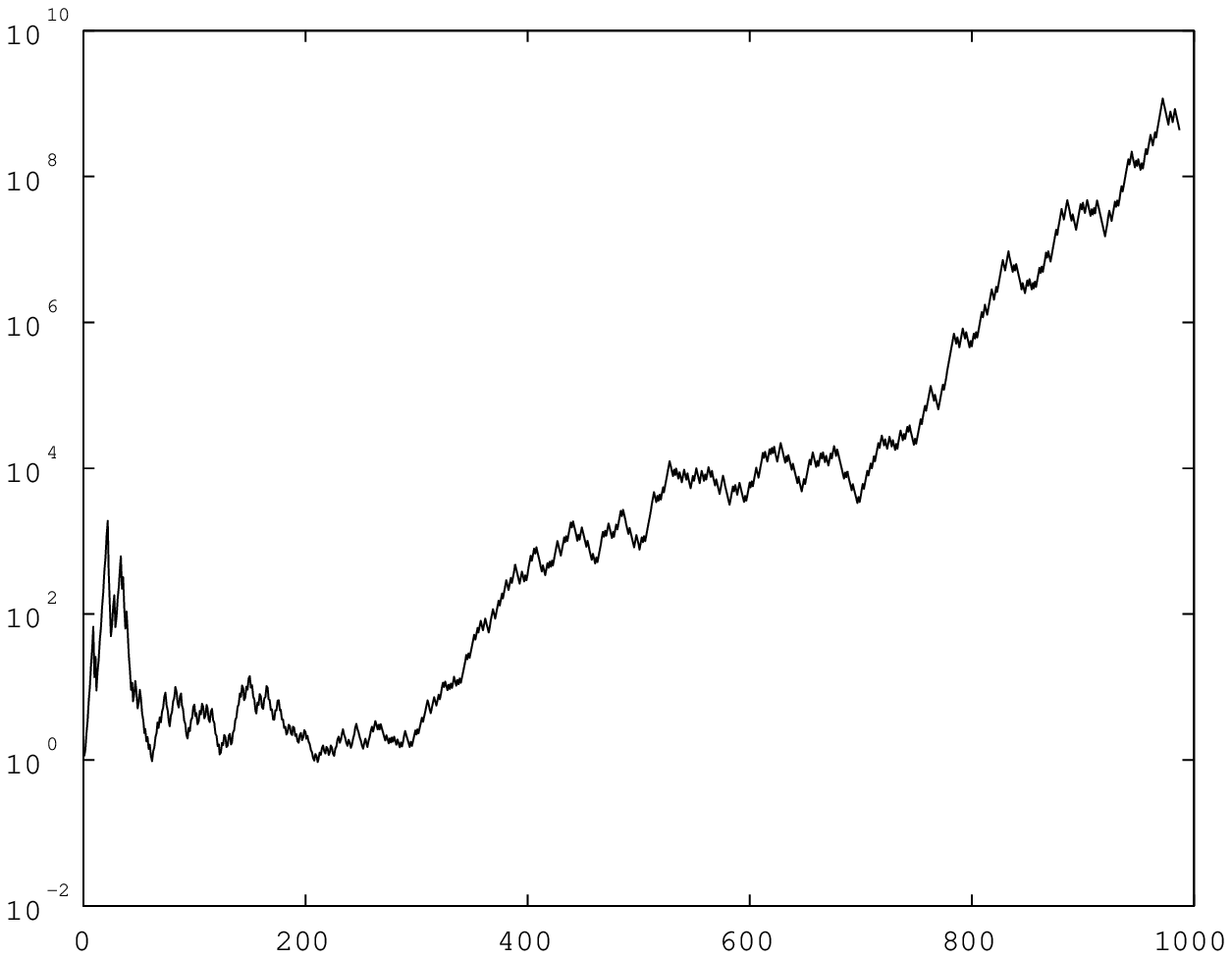}
  \caption{Case 2 with strategy 2}
  \label{st2_cs2}
 \end{minipage}
 \begin{minipage}{0.5\hsize}
   \includegraphics[width=85mm]{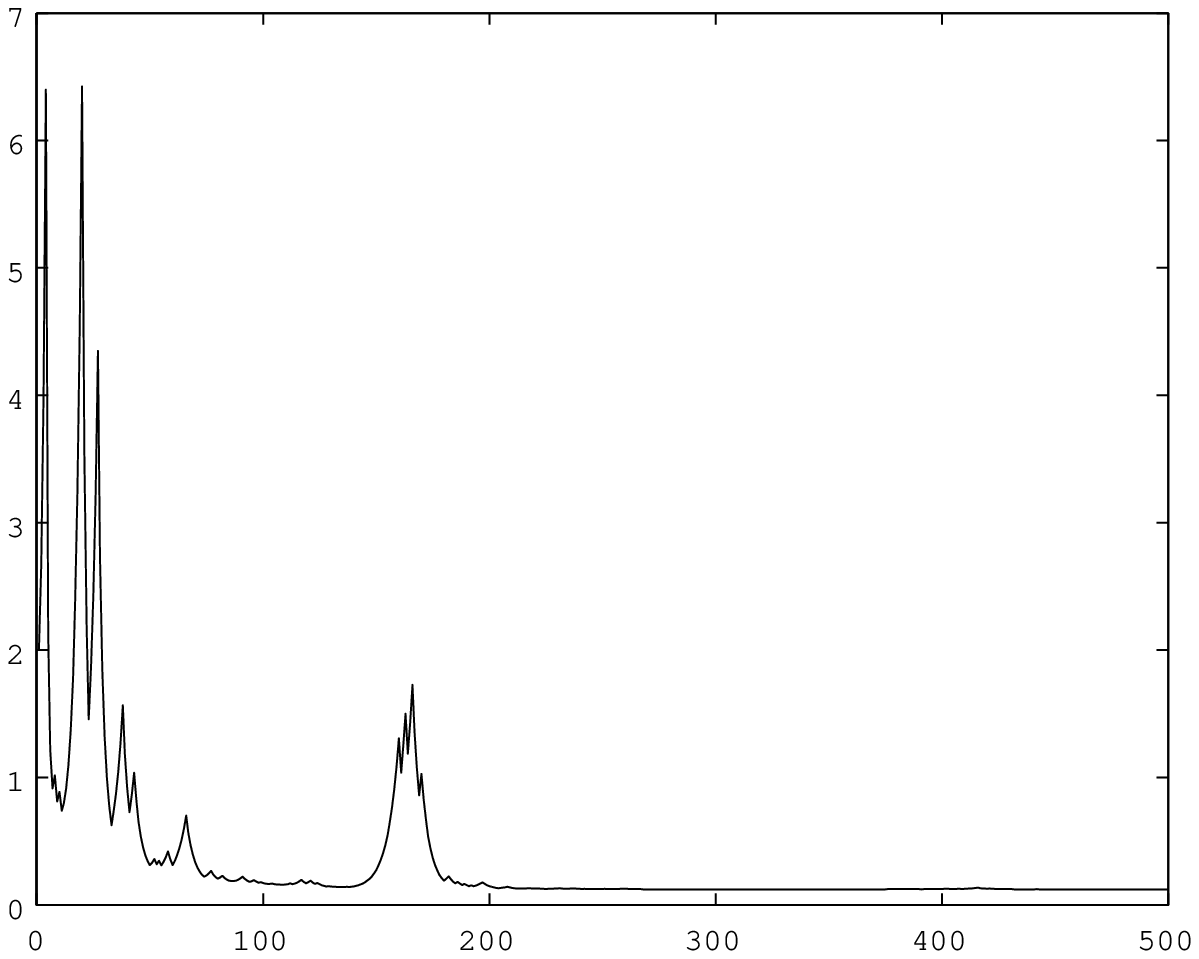}
  \caption{Case 3 with strategy 2}
  \label{st2_cs3}
 \end{minipage}
\end{figure}

\begin{figure}[htbp]
\begin{center}
\includegraphics[width=90mm]{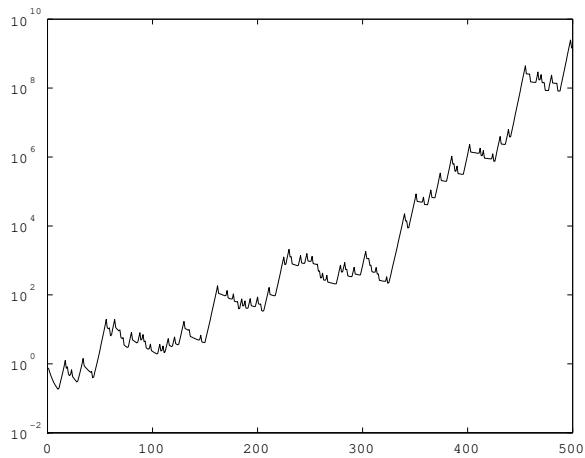}
\caption{Case 3 with strategy 3}
\label{st3_cs3}
\end{center}
\end{figure}

As shown in Figure \ref{st1_cs1} and Figure \ref{st1_cs2}, we can beat Reality by strategy 1 only in case 1. So we improve our strategy and apply strategy 2 to case 2. We can see from Figure \ref{st2_cs2} and Figure \ref{st2_cs3} that strategy 2 can work well in case 2 but still not effective in case 3. Finally, we use strategy 3 in case 3 and observe that it shows a good result for Skeptic in Figure \ref{st3_cs3}.

From these simulations, we see that Skeptic can beat Reality with more flexible strategy utilizing more side information.

\subsection{Betting against probability of precipitation by the Japan Meteorological Agency}
\label{subsec:JMA}

Now we apply our strategy to probability of precipitation 
provided by the Japan Meteorological Agency.
We collected the forecast probabilities for the Tokyo area from archives of the morning edition of 
the Mainichi Daily News and the actual weather data on 9:00 and 15:00 of each day for Tokyo area 
from {\tt http://www.weather-eye.com/} for the period of three years from 
January 1, 2009 to December 31, 2011.  We counted a day as rainy if the data on this site records
rain on 9:00 or on 15:00 of that day in Tokyo area.

The forecast probability $p_n$ is only announced as multiples of 10\% (i.e.\ $0\%, 10\%, \dots, 90\%, \allowbreak 100\%$) by JMA.  The data are summarized in Table \ref{data}.
$p_n$ represents the probability of precipitation on day $n$ and $x_n$ indicates the actual precipitation.
Actual ratio is calculated from the ratio of the number of rainy days to all days for a given value of $p_n$.

\begin{table}[htbp]
\caption{Actual ratio of rainy days}
\label{data}
\begin{center}
\begin{tabular}{|l|c|r|r|} \hline
$p_n$(\%)&$x_n=1$&$x_n=0$&Actual Ratio(\%) \\ \hline
0&1&61&1.6 \\
10&10&324&3.0 \\
20&24&193&11.1 \\
30&36&117&23.5 \\
40&20&26&43.5 \\
50&67&56&54.5 \\
60&38&14&73.1 \\
70&36&7&85.7 \\
80&36&4&90.0 \\
90&22&1&95.6 \\
100&3&0&100 \\ \hline
\end{tabular}
\end{center}
\end{table}
\begin{figure}[htbp]
\begin{center}
\includegraphics[width=12cm]{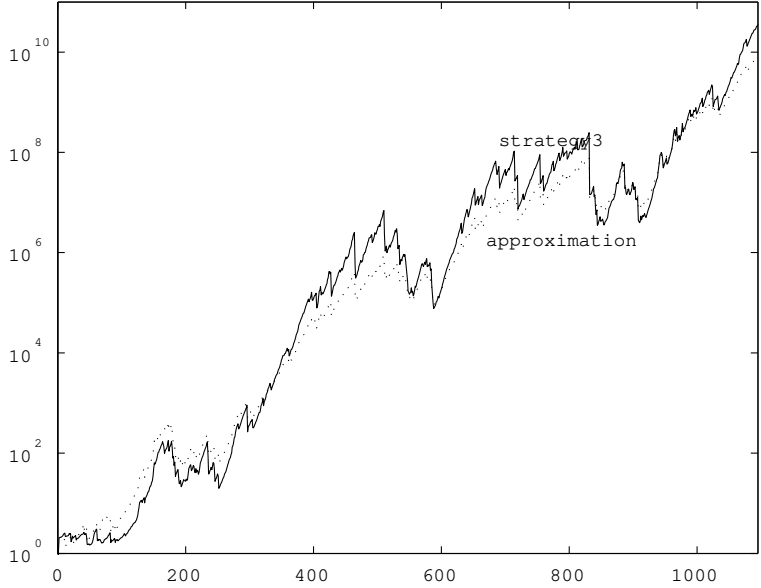}
\caption{Beating JMA by strategy 3 with $\beta$ uniform over $[0,2]$}
\label{data_st3}
\end{center}
\end{figure}

The distinct feature of the prediction by JMA is that  that 
$p_n$ tends to be closer to 50\% than the actual ratio.  For example,
when JMA announces $p_n=20\%$, the actual ratio is only $11.1\%$.
Similarly when JMA announces $p_n=80\%$, the actual ratio is $90\%$.
Hence JMA has the tendency of avoiding clear-cut forecasts.

In the hindsight strategy, the value of $\beta$, which is a coefficient for $\log(p_n/(1-p_n))$ 
in strategy 3 is close to 1.5.  Hence we modified strategy 3 of the previous section, so that
the prior for $\beta$ is uniform between 0 and 2.
We also substituted $p_n=1\%$ and $p_n=99\%$ for $p_n=0\%$ and  $p_n=100\%$, respectively,
because our strategy is not defined for $p_n=0\%$ or $100\%$.
Figure \ref{data_st3} shows %
the behavior of strategy 3 and the approximation $\cS_n\trns \cV_n^{-1} \cS_n/2$.
We see that our strategy works very well against JMA by exploiting its tendency 
of avoiding clear-cut forecasts.  It is also of interest that the capital process shows a
seasonal fluctuation and it does not perform well for the rainy season (June and July) in Tokyo area.

\section{Summary and discussion}
\label{sec:summary}

In this paper we proposed a Bayesian logistic betting strategy in the binary probability forecasting
game with side information (BPFSI).  We proved some theoretical results and showed good performance
of our strategy against probability forecasting by Japanese Meteorological Agency.

Here we discuss some topics for further investigation.  

We considered implications of a single Bayesian logistic betting strategy
in BPFSI. We can also take a look at the sequential optimizing strategy (SOS) of \cite{sos}
in BPFSI. Under the condition $\mle \rightarrow 0$, Bayesian strategy and 
SOS should behave in the same way.  However
we could not succeed to prove weak forcing of $\mle \rightarrow 0$ by SOS alone.

For the case of $d=1$ we could employ approaches of \cite{miyabe-takemura} to prove results similar to Theorem
\ref{thm:main}.  In \cite{miyabe-takemura} we also discussed Reality's strategies.  It is of 
interest to study strategies of Forecaster or Reality in the binary probability forecasting
game with side information.  Defensive forecasting 
(\cite{defensive-forecasting}, \cite{df-linear})
can be considered as a strategy of Forecaster.

We extended the binary probability forecasting game by including side information.
In our formulation side information $c_n$ is announced by Forecaster and 
in our logistic betting strategy $c_n$ is used as regressors in a logistic regression.
However Skeptic can use any transformation of $c_n$ in his strategy.  In this sense, it
might be more natural to formulate the game, where $c_n$ is announced by Skeptic.
Binary probability forecasting game is often considered from the viewpoint of 
prequential probability (\cite{dawid-vovk}) and leads to the notion of randomness
of the sequence $p_1 x_1 p_2 x_2 \dots$ (\cite{vovk-shen}, \cite{miyabe-superfarthingale}).
From the viewpoint of prequential probability it might also be natural to consider side information
$c_n$ as a part of moves by Skeptic for testing the randomness of  $p_1 x_1 p_2 x_2 \dots$.

We assumed multidimensional $c_n$.  However from the viewpoint of
game-theoretic probability, we do not lose much generality by restricting $c_n$ to be a scalar, since
if Skeptic can weakly force events $E_1,\dots, E_d$ then he can weakly force $E_1 \cap \dots \cap E_d$.
By the same reasoning we can also consider $d=\infty$, because if 
Skeptic can weakly force $E_1,E_2, \dots,$ then he can weakly force $\cap_{i=1}^\infty E_i$.
Interpretation and formulation of side information in game-theoretic probability needs
further investigation.

\bibliographystyle{abbrv}
\bibliography{logistic}

\def\cprime{$'$}
\begin{thebibliography}{10}

\bibitem{arnold1989}
V.~I. Arnol{\cprime}d.
\newblock {\em Mathematical {M}ethods of {C}lassical {M}echanics}, volume~60 of
  {\em Graduate Texts in Mathematics}.
\newblock Springer-Verlag, New York, second edition, 1989.
\newblock Translated from the Russian by K. Vogtmann and A. Weinstein.

\bibitem{barndorff-nielsen}
O.~Barndorff-Nielsen.
\newblock {\em Information and {E}xponential {F}amilies}.
\newblock John Wiley \& Sons Inc., New York, 1978.
\newblock Wiley Series in Probability and Mathematical Statistics.

\bibitem{cover-universal-portfolios}
T.~M. Cover.
\newblock Universal portfolios.
\newblock {\em Math. Finance}, 1(1):1--29, 1991.

\bibitem{cover-ordentlich}
T.~M. Cover and E.~Ordentlich.
\newblock Universal portfolios with side information.
\newblock {\em IEEE Trans. Inform. Theory}, 42(2):348--363, 1996.

\bibitem{cover-thomas-2nd}
T.~M. Cover and J.~A. Thomas.
\newblock {\em Elements of {I}nformation {T}heory}.
\newblock Wiley, Hoboken, NJ, second edition, 2006.

\bibitem{dawid-1986}
A.~P. Dawid.
\newblock Probability forecasting.
\newblock In {\em Encyclopedia of Statistical Sciences}, volume~7, pages
  210--218. Wiley, New York, 1986.

\bibitem{dawid-vovk}
A.~P. Dawid and V.~G. Vovk.
\newblock Prequential probability: principles and properties.
\newblock {\em Bernoulli}, 5(1):125--162, 1999.

\bibitem{jensen}
J.~L. Jensen.
\newblock {\em Saddlepoint {A}pproximations}.
\newblock Oxford University Press, Oxford, 1995.

\bibitem{kelly-1956}
J.~L. Kelly, Jr.
\newblock A new interpretation of information rate.
\newblock {\em Bell. System Tech. J.}, 35:917--926, 1956.

\bibitem{ktt-2008saa}
M.~Kumon, A.~Takemura, and K.~Takeuchi.
\newblock Capital process and optimality properties of a {B}ayesian skeptic in
  coin-tossing games.
\newblock {\em Stoch. Anal. Appl.}, 26(6):1161--1180, 2008.

\bibitem{sos}
M.~Kumon, A.~Takemura, and K.~Takeuchi.
\newblock Sequential optimizing strategy in multi-dimensional bounded
  forecasting games.
\newblock {\em Stochastic Process. Appl.}, 121:155--183, 2011.

\bibitem{kelly-criterion-book}
L.~C. MacLean, E.~O. Thorp, and W.~T. Ziemba.
\newblock {\em The Kelly {C}apital {G}rowth {I}nvestment {C}riterion: {T}heory
  and {P}ractice}.
\newblock World Scientific handbook in financial economic series. World
  Scientific, 2011.

\bibitem{miyabe-superfarthingale}
K.~Miyabe.
\newblock An optimal superfarthingale and its convergence over a computable
  topological space, 2011.
\newblock To appear in Proceedings of Solomonoff 85th Memorial Conference at
  Melbourne, Australia.

\bibitem{miyabe-takemura}
K.~Miyabe and A.~Takemura.
\newblock Convergence of random series and the rate of convergence of the
  strong law of large numbers in game-theoretic probability.
\newblock {\em Stochastic Process. Appl.}, 122(1):1--30, 2012.

\bibitem{shafer-etal-test-martingales}
G.~Shafer, A.~Shen, N.~Vereshchagin, and V.~Vovk.
\newblock Test martingales, {B}ayes factors and $p$-values.
\newblock {\em Statist. Sci.}, 26(1):84--101, 2011.

\bibitem{ShaVov01}
G.~Shafer and V.~Vovk.
\newblock {\em Probability and Finance: It's Only a Game!}
\newblock Wiley, 2001.

\bibitem{ktt-new-procedure}
K.~Takeuchi, A.~Takemura, and M.~Kumon.
\newblock New procedures for testing whether stock price processes are
  martingales.
\newblock {\em Computational Economics}, 37(1):67--88, 2010.

\bibitem{df-linear}
V.~Vovk, I.~Nouretdinov, A.~Takemura, and G.~Shafer.
\newblock Defensive forecasting for linear protocols.
\newblock In H.~S.Jain and E.Tomita, editors, {\em Proceedings of the 16th
  international conference on algorithmic learning theory}, number 3734 in
  LNAI, pages 459--473, 2005.

\bibitem{vovk-shen}
V.~Vovk and A.~Shen.
\newblock Prequential randomness and probability.
\newblock {\em Theoret. Comput. Sci.}, 411(29-30):2632--2646, 2010.

\bibitem{defensive-forecasting}
V.~Vovk, A.~Takemura, and G.~Shafer.
\newblock Defensive {F}orecasting.
\newblock In R.G.Cowell and Z.Ghahramani, editors, {\em Proceedings of the
  tenth international workshop on artificial intelligence and statistics},
  pages 365--372, 2005.

\end{thebibliography}

\end{document}